\documentclass[usenames,dvipsnames]{amsart}
\usepackage{amsmath,amssymb,amsthm}
\usepackage{comment}
\usepackage{tikz}
\usetikzlibrary{matrix}
\usetikzlibrary{shapes}
\usetikzlibrary{decorations.pathreplacing, arrows.meta}
\usepackage{enumerate}
\usepackage{mathtools}

\newcommand{\uhr}{\upharpoonright}
\newcommand{\concat}{^\smallfrown}
\tikzstyle{block} = [draw,minimum size=1em]
\newcommand{\RCA}{\mathsf{RCA}_0}
\newcommand{\ACA}{\mathsf{ACA}_0}
\newcommand{\WKL}{\mathsf{WKL}_0}
\newcommand{\WWKL}{\mathsf{WWKL_0}}
\newcommand{\ATR}{\mathsf{ATR}_0}
\newcommand{\CDM}{\mathsf{CD}\text{-}\mathsf{M}}

\newcommand{\CDPB}{\mathsf{CD}\text{-}\mathsf{PB}}

\newcommand{\LwCA}{\mathsf{L_{\omega_1,\omega}}\text{-}\mathsf{CA}}

\newcommand{\makeset}[1]{|#1|}
\newcommand{\Decorate}{\operatorname{Decorate}}

\newcommand{\kO}{\mathcal O}

\newtheorem{theorem}{Theorem}[section]
\newtheorem{definition}[theorem]{Definition}
\newtheorem{proposition}[theorem]{Proposition}
\newtheorem{lemma}[theorem]{Lemma}
\newtheorem{corollary}[theorem]{Corollary}

\title{Completely determined Borel sets and measurability}

\author[L.\ Westrick]{Linda Westrick}
\address{Department of Mathematics\\
Penn State University\\
University Park, Pennsylvania U.S.A.}
\email{westrick@psu.edu}

\thanks{The author was supported by grant DMS-1854107 from the National Science Foundation of the United States and by the Cada R. and Susan Wynn Grove Early Career Professorship in Mathematics.  Part of the work was done during the author's IMS-supported visit to the Institute for Mathematical Sciences, National University of Singapore in 2019.}

\begin{document}

\begin{abstract}
We consider the reverse math strength of the statement
$\CDM$:
``Every completely determined Borel set is measurable.''
Over $\WWKL$, we obtain the following results analogous to the previously
studied category case:
\begin{enumerate}
\item $\CDM$ lies strictly between $\ATR$ and $\LwCA$.
\item Whenever 
$M\subseteq 2^\omega$ is the second-order part of an 
$\omega$-model of $\CDM$, then for every $Z \in M$,
there is a $R \in M$ such that $R$ is $\Delta^1_1$-random 
relative to $Z$.
\end{enumerate}  On the other hand, without $\WWKL$, 
all sets have measure zero (as measured according to $\CDM$), and
it follows vacuously that $\neg \WWKL$ 
implies $\CDM$ over $\RCA$.
\end{abstract}

\maketitle

\section{Introduction}

The notion of a \emph{completely determined Borel set} was 
introduced in \cite{ADMSW} to permit the reverse mathematics
analysis of weak principles involving Borel sets.  In the standard
treatment of Borel sets in reverse mathematics \cite{sosa}, 
a Borel set is any well-founded 
tree $T$ whose leaves are labeled with clopen sets and whose
interior nodes are labeled with intersections or unions.  A 
real $X \in 2^\omega$ is then said to belong to the set 
coded by $T$ if and only if there is an \emph{evaluation map},
a function $f:T\rightarrow \{0,1\}$ such that $f(\sigma) = 1$ 
if and only if $X$ is in the set coded by 
$T_\sigma := \{\tau : \sigma\concat \tau \in T\}$.  While 
\emph{arithmetic transfinite recursion} ($\ATR$) suffices 
to construct evaluation maps for each $X$, in general 
it is also required.  As a result, most principles concerning an
arbitrary Borel set reverse to $\ATR$ 
simply because most such principles have a conclusion 
that presupposes an element $X$ in the Borel set.

An exception was encountered by \cite{DFSW} in their 
analysis of the Borel dual Ramsey theorem.  The hypothesis 
of this theorem posits $\ell$-many Borel sets whose union 
is the entire space.  In order to say the union is the entire space, 
the existence of evaluation maps for each $X$ must 
be a part of the hypothesis.  That is, an instance of the 
Borel dual Ramsey theorem is not well-defined unless 
the given Borel sets are \emph{completely determined}, 
meaning that each $X$ has an evaluation map.  

This 
example fueled the idea that the lack of interesting reversals
for weak principles involving Borel sets could be remedied 
by restricting attention to completely determined Borel sets.  
This was borne out in \cite{ADMSW}, in which the following 
was proven about the principle $\CDPB$: ``Every completely 
determined Borel set has the property of Baire.''

\begin{theorem}[\cite{ADMSW}]
The principle $\CDPB$ is strictly weaker than $\ATR$.  Every 
$\omega$-model $\mathcal M$ of $\CDPB$ is closed under hyperarithmetic
reduction, and for every $Z \in M$, there is some $G \in \mathcal M$ 
that is $\Delta^1_1(Z)$-generic. 
\end{theorem}

In this paper we do the same for the principle ``every Borel 
set is measurable.''  Similar results are obtained by similar 
methods.  The only new twist is the need to  
work with
an appropriate meaning of  ``measurable'' for a Borel
set; there are several candidates.  
This delicate task has already been undertaken 
by Simpson, X. Yu, Brown, Giusto and others (see for 
example \cite[Chapter X]{sosa}, \cite{YuX1993}, \cite{YuX1994},
and \cite{BrownGiustoSimpson2002}).  We summarize 
their work and give the sometimes more detailed versions 
of the results needed for our
 application.

We then define the principle $\CDM$: ``Every completely 
determined Borel set is measurable.''  We show that 
$\CDM$ follows from $\neg \WWKL$ (for the simple reason 
that $\neg \WWKL$ implies the Cantor space has measure 0,
and thus every subset of it is also measure 0).
On the other hand, working over $\WWKL$, we obtain 
results similar to the category case. 

In \cite{ADMSW}, a model was constructed in which a Baire approximation 
to a given completely determined 
Borel set $B$ was obtained without $\ATR$ by polling 
$\Sigma^1_1(B)$-generics about their membership in $B$.  We 
do essentially the same to construct a proof of measurability of 
a given completely determined set $B$, but 
using $\Pi^1_1(B)$-randoms.  The result of this polling is exactly 
an element $f \in L^1(2^\omega)$, so no translation is required
to obtain a code for 
a measurable set as defined in \cite[Chapter X]{sosa}.
The main 
results of this paper are as follows.

\begin{theorem}
The principle $\WWKL+\CDM$ is strictly weaker than $\ATR$.  Every 
$\omega$-model $\mathcal M$ of $\WWKL+\CDM$ is closed under hyperarithmetic
reduction, and for every $Z \in \mathcal M$, there is some $R \in \mathcal M$ 
that is $\Delta^1_1(Z)$-random. 
\end{theorem}

The related topic of measure-theoretic regularity (abbreviated MTR) 
was investigated by
Simpson in \cite{Simpson2009}.  By definition, an 
$\omega$-model $\mathcal M$ 
is an MTR-model if every set that is effectively Borel in a
parameter $X$ from $\mathcal M$ contains a $\Sigma^0_2(Y)$
subset of the same measure, for some $Y \in \mathcal M$.  The 
above theorem implies that every $\omega$-model of $\WWKL+\CDM$
is an MTR-model, because to be an MTR-model it suffices to be closed 
under hyperarithmetic reduction.  However, there are MTR-models 
which satisfy, for example, $\WWKL$ but not $\WKL$ 
(\cite[Theorem 7.4]{Simpson2009}).  So being an MTR-model is a strictly 
weaker notion than being an $\omega$-model of $\WWKL +\CDM$.

These results were first presented by the author at 
the Institute for Mathematical Sciences workshop Higher Recursion 
Theory and Set Theory in 2019, using a version of Proposition \ref{joinACA} 
to quickly move the base theory to $\ACA$, and using an ad hoc 
notion of a ``function measuring a set'' which was later found to 
essentially coincide with the notion of a measurable characteristic 
function previously proposed by Simpson and several of his collaborators.  
The author would like to thank Steve Simpson for his suggestion to 
lower the base theory and for bringing that 
connection to light.  Thanks go also to the anonymous referee who 
provided further helpful suggestions.
Finally, the author would like to thank Ted Slaman, 
her PhD advisor, for his support and mentorship, his good 
humor and sound principles, and his excellent body of 
research which this volume celebrates.

\section{Notation and Preliminaries}

We use the notation and conventions of \cite{ADMSW}.  In that paper, 
much more background and context can be found in the introduction. 
The $e$th Turing functional is denoted $\Phi_e$.
Elements of $\omega^{<\omega}$ are denoted by $\sigma, \tau$ 
and elements of $2^{<\omega}$ by $p,q$.  We write $\sigma\preceq \tau$ to 
indicate that $\sigma$ is an initial segment of $\tau$, with $\prec$ 
if $\sigma\neq \tau$.    
For $p \in 2^{<\omega}$, the notation $[p]$ refers to the cylinder 
$\{X \in 2^\omega : p \prec X\}$.
The empty 
string is denoted by $\lambda$.  A string with a single component 
of value $n\in\omega$ is denoted by $\langle n \rangle$.  String 
concatenation is denoted by $\sigma \tau$.  Usually 
we write $\sigma n$ instead of the more technically correct 
but uglier $\sigma \langle n \rangle$.  

If $U$ is a set of strings (for example, a tree, or a coded 
open subset of $2^\omega$), and $\sigma$ is any string, we write 
$\sigma \concat U$ to mean $\{\sigma  \tau : \tau \in U\}$.
If $T$ is a tree and $\sigma \in T$, we write
$T_\sigma$ to mean $\{\tau : \sigma \tau \in T\}$, 
and if $\langle n \rangle  \in T$, we write 
$T_n$ to mean $\{\tau : n\tau \in T\}$.

We assume familiarity with reverse mathematics, in particular the 
systems $\RCA$, $\WWKL$, $\ACA$ and $\ATR$.  We note that 
effective transfinite recursion and 
arithmetic transfinite induction can be carried out in $\ACA$.  
We identify an $\omega$-model $\mathcal M$ of second order arithmetic 
with its second-order part, writing $X \in \mathcal M$ to mean 
that $X$ is an element of the second-order part of $\mathcal M$.

We assume familiarity with ordinal notations and pseudo-ordinals.
Kleene's O is denoted by $\kO$.  The relation $<_\ast$ is the 
transitive closure of
the relation defined by 
$1<_\ast x$ if $x\neq 1$, $x <_\ast 2^x$, 
and $\Phi_e(n) <_\ast 3\cdot 5^e$. We will not distinguish 
between ordinals and their notations.  Additionally, if $b \in \kO$, we 
write $b+1$ for the successor of $b$ (rather than the 
more technically correct but cumbersome $2^b$) and $b+O(1)$ 
for the outcome of taking some fixed constant number of successors 
of $b$.  If $b \in \kO$ the unique jump hierarchy on $b$ is 
denoted $H_b$.  All these concepts can be relativized to 
an oracle $Z$.  Kleene's $\kO$ also has a $\Sigma^1_1$ superset 
$\kO^\ast$, defined as the intersection 
of all $X \in HYP$ such that $1 \in X$, $a \in X \implies 2^a \in X$, and 
$$\forall n [\Phi_e(n)\in X \text{ and } \Phi_e(n) <_\ast \Phi_e(n+1)] \implies 3\cdot 5^e \in X.$$
Observe also that $\kO$ is contained in $\kO^\ast$.  The elements 
of $\kO^\ast \setminus \kO$ are called pseudo-ordinals.
For more details, see the introduction of \cite{ADMSW}.

A $T\subseteq \omega^{<\omega}$ is well-founded if it has no 
infinite path.  If $T$ is any tree, 
and $\rho : T\rightarrow \kO^\ast$, we say that  $\rho$ 
\emph{ranks} $T$ if 
for all $\sigma$ and $n$ 
such that $\sigma\concat n \in T$, we have 
$\rho(\sigma \concat n) <_\ast \rho(\sigma)$, 
and 
 for each leaf $\sigma \in T$, $\rho(\sigma) = 1$. 
If $T$ is ranked by $\rho$ and $\rho(\lambda) = a$, 
we say that $T$ is \emph{$a$-ranked} by $\rho$.  
If $a \in \kO$ and $T$ is $a$-ranked then $T$ is well-founded, 
but it is possible and useful for an ill-founded tree to be ranked 
by a pseudo-ordinal.  A tree $T$ is \emph{alternating} if 
whenever $\sigma \in T$ is a $\bigcap$, then 
each $\sigma n \in T$ is either a $\bigcup$ or a leaf, 
and similarly if $\sigma \in T$ is a $\bigcup$, then each 
$\sigma n \in T$ is either a $\bigcap$ or a leaf.

A labeled \emph{Borel code} is a well-founded tree 
$T\subseteq \omega^{<\omega}$ whose leaves are labeled 
by basic open sets or their complements, and whose inner nodes
are labeled by $\bigcup$ or $\bigcap$.  The Borel set associated to a 
Borel code is defined by induction, 
interpreting the labels in the obvious way.
Any Borel set 
can be represented this way, by applying DeMorgan's laws to push
complementation out to the leaves.
A formula of $L_{\omega_1,\omega}$ is a well-founded tree whose 
interior nodes are labeled with $\bigwedge$ (conjunction) and $\bigvee$
(disjunction) and whose leaves are labeled with the symbols 
{\tt true} or {\tt false}.  

There is a computable procedure which, 
for any $b \in \kO$ and any $n \in \omega$, 
outputs a $b+O(1)$-ranked alternating
formula of $L_{\omega_1,\omega}$ which holds true if 
and only if $n \in H_b$.

If $T$ is a labeled Borel code and $X \in 2^\omega$,
an \emph{evaluation map}
for $X \in T$ is a function
$f:T\rightarrow \{0,1\}$ such that 
\begin{itemize}
  \item If $\sigma$ is a leaf, $f(\sigma) = 1$ if and only if $X$
    is in the clopen set coded by $\ell(\sigma)$.
  \item If $\sigma$ is a union node, $f(\sigma) = 1$ if and only
    if $f(\sigma\concat n) = 1$ for some $n\in \omega$.
  \item If $\sigma$ is an intersection node, $f(\sigma) = 1$ if and
    only if $f(\sigma\concat n) = 1$ for all $n \in \omega$.
  \end{itemize}
We say that $X$ is in the set coded by $T$, denoted $X \in \makeset T$, 
if there is an evaluation map $f$ for $X$ in $T$ such that 
$f(\lambda) = 1$.  Note that $X \in \makeset T$ is a 
$\Sigma^1_1$ statement.  In $\ACA$, evaluation maps are unique 
when they exist.  If $T$ is ill-founded, the notation $|T|$ may not have 
meaning outside of a given model.  If $T$ is a truly
well-founded Borel code, 
we do use $|T|$ outside of the context of a model 
to denote the elements of the set that $T$ codes.

A Borel code $T$ is \emph{completely determined} if every $X \in 2^\omega$ 
has an evaluation map in $T$.  A formula $\phi$ of $L_{\omega_1,\omega}$ 
is \emph{completely determined} if there is map $f:\phi\rightarrow \{\tt true, false\}$
that agrees with $\phi$ on the leaves and satisfies the logic of 
$\phi$ at interior nodes.  The principle $\LwCA$ states that whenever 
$\langle \phi_n \rangle_{n\in \omega}$ is a sequence of completely 
determined formulas of $L_{\omega_1,\omega}$, then $\{n : \phi_n \text{ is true}\}$ 
exists.

We assume familiarity with higher randomness.  The key theorems we need are:
\begin{theorem}[\cite{Stern1973,Stern1975}]\label{thm:nomoveomega1}
A real $R \in 2^\omega$ is $\Pi^1_1$-random if and only if it is $\Delta^1_1$-random 
and $\omega_1^R = \omega_1^{ck}$.
\end{theorem}
\begin{theorem}[\cite{HjorthNies2007}]
For $R_0,R_1 \in 2^\omega$, we have $R_0 \oplus R_1$ is $\Pi^1_1$-random 
if and only if $R_0$ and $R_1$ are relatively $\Pi^1_1$-random.
\end{theorem}
\begin{theorem}[\cite{ChongNiesYu2008}]
If $R_0\oplus R_1$ is $\Pi^1_1$-random, then $\Delta^1_1(R_0) \cap \Delta^1_1(R_1) = \Delta^1_1$.
\end{theorem}

\section{Measure theory in reverse mathematics}

Historically, measure theory developed as a third-order theory.
Classically, a measure is a set function from a $\sigma$-algebra of subsets of a
space to the non-negative reals.  Therefore, although much of 
measure theory can be developed within second-order arithmetic,
this development has required some care and some non-trivial choices.
We now summarize work of Simpson, X. Yu, Brown, and Giusto
\cite{YuX1990, SimpsonYu1990, YuX1993, YuX1994, BrownGiustoSimpson2002, sosa},
in which this development took place.

In the context of second-order arithmetic, all the relevant 
information about a measure space $(X,\mu,\mathcal S)$ 
is already contained in the values 
that $\mu$ takes on an algebra which generates $\mathcal S$
as a $\sigma$-algebra.  When $X$ is a separable complete metric space 
space and $\mathcal S$ is the Borel sets, a countable generating algebra
is naturally obtained by 
taking all finite Boolean combinations of basic open sets.  In the case 
of Cantor space $2^\omega$, this approach works out very cleanly 
because the basic open sets (and thus all elements of the 
generating algebra) are clopen.  However, for an
arbitrary separable complete metric space, a problem arises.
What if there is an atom on the boundary of a basic open set $U$?
Is it fair to ask that our encoding of a measure $\mu$
be able to precisely compute $\mu(U)$ and $\mu(U^c)$?
(Because a typical open set $V$ can only be represented as 
an infinite enumeration of its basic open subsets, its 
measure $\mu(V)$ would be at best 
c.e., not computable, in a description of $\mu$ and $V$.)
Another way of asking the same question is: for the purposes
of constructive mathematics, what is a suitable 
topology to put on the space of Borel measures on $X$?

When $X$ is Cantor space, a popular representation
choice has been
to to name a measure $\mu$ 
with a function from $2^{<\omega}$ to $\mathbb R$ which records 
the measure of each basic clopen set (see for example 
\cite{DayMiller2013}). This representation induces the so-called 
weak topology on the space of probability measures on $X$
(see for example \cite[Definition 8.2.1]{Bogachev-V2}). 
This is the same topology induced by the 
Prohorov metric (see for example \cite[Theorem 8.3.2]{Bogachev-V2}), 
and also coincides with the weak-$\ast$ 
topology on $C(X)^\ast$ (see the discussion following 
Definition 8.2.1 in \cite{Bogachev-V2}).
Restricting attention to probability measures on compact 
complete separable metric spaces,
Yu also settled on the same topology in \cite{YuX1993}, and 
made the following definition.

\begin{definition}\label{def:1} Let $X$ be a compact complete separable metric space.  A 
\emph{Borel probability measure} $\mu$ on $X$ is 
a bounded positive 
linear functional $\mu:C(X) \rightarrow \mathbb R$ with 
$\mu(1) = 1$.  
\end{definition}
Here $C(X)$ denotes the Banach space of 
continuous real-valued functions on $X$ with the 
supremum norm, and $1 \in C(X)$ denotes the constant function.
Care is required in the definition of $C(X)$.  It is 
not simply the collection of continuous function on $X$ 
equipped with the supremum norm, because in weak subsystems 
of second-order arithmetic, a continuous function on 
a compact space $X$ need not have a supremum.  Instead, 
$C(X)$ is defined as a complete separable metric space by 
choosing a particularly well-behaved collection of continuous 
functions to be the dense subset.  The details are given in
\cite[Exercise 4.2.13]{sosa}, in which it is also established that 
$C(X)$ consists of precisely those continuous 
functions from $X$ to $\mathbb R$ which also possess a 
modulus of uniform continuity.  Therefore, while a measure $\mu$ on $X$ 
is defined by specifying how to integrate elements of $C(X)$ 
with respect to $\mu$, it does not follow that every 
continuous function on $X$ is $\mu$-integrable; only those with a 
modulus of uniform continuity come with this guarantee.

An unavoidable drawback to Definition \ref{def:1} is that it puts a small 
distance between the definition of a measure and its basic 
function of assigning sizes to sets.  Therefore, it is necessary to 
make a further definition for ``the measure of an open set'' 
(and subsequently a further definition for the measure of an 
arithmetic set, etc. leading up to the notion of a measurable set).
At each point of definition, a choice arises: should the measure 
assignment be 
\emph{intensional} (depending only on the \emph{description} 
of the set in question)
or \emph{extensional} (depending on only on the \emph{membership} of the set 
in question)?

To understand the tension here, consider that if 
$U$ is any component of a universal 
Martin-L\"of test in Cantor space with its usual fair-coin measure, 
then statement $U = 2^\omega$ holds in $REC$.  Thus in 
$REC$, we cannot simultaneously have both of these two desirable properties:
\begin{enumerate}
\item If $S\subseteq 2^{<\omega}$ is prefix-free, then
    $\mu\left(\bigcup_{\sigma \in S} [\sigma]\right) = \sum_{\sigma \in S} 2^{-|\sigma|}$
\item If $A = B$ then $\mu(A) = \mu(B)$.
\end{enumerate}
Note that the first is an intensional property and the second is an extensional 
property.
Although both are clearly wanted, the second seems more essential.
Thus the extensional definition for the measure of an open set 
is the one which 
appears in \cite{sosa}.

\begin{definition}[$\RCA$]\label{def:2}
Let $\mu$ be a Borel probability measure on $X$.  Let $U$ be an open 
subset of $X$.  The \emph{$\mu$-measure} of $U$ is defined as 
$$\mu(U) = \sup \{ \mu(f) : f \in C(X), 0 \leq f \leq 1, f(x) = 0 \text{ for } x\in X \setminus U\}.$$
\end{definition}

In the absence of $\ACA$, this supremum may not exist as a number, 
but statements about $\mu(U)$ may still be made in weaker systems 
by simply substituting the above definition of $\mu(U)$ in any 
sentence which makes a claim about this quantity.  For example, 
it holds in $\RCA$ that $U\subseteq V$ implies that $\mu(U)\leq \mu(V)$.
Such statements are said to hold in a ``virtual'' 
or ``comparative'' sense.

Observe that this extensional 
definition also gives the ``right'' values on Cantor space with the fair 
coin measure when $U$ is a finite union of non-intersecting 
cylinders $U= \cup_{i<n} [p_i]$.  That is, 
$\mu(\cup_{i<n} [p_i]) = \sum_{i<n} 2^{-|p_i|}$.

On the other hand, in $\RCA$ we can always assume that open subsets 
of Cantor space are given by prefix-free enumerations of elements 
of $2^{<\omega}$, so we can also give the following intensional definition of 
measure of an open set in Cantor space:

\begin{definition}[$\RCA$]
If $U$ is an open subset of $2^\omega$ given by $U = \bigcup_{i<\omega} [p_i]$,
where each $p_i \in 2^{<\omega}$ and
where $\{p_i : i \in \omega\}$ is prefix-free, then define the \emph{intensional 
measure} of $U$ by $\mu_I(U) = \sum_i 2^{-|p_i|}$.
\end{definition}

The intensional and extensional definitions fully coincide under
$\WWKL$.
\begin{theorem}[\cite{SimpsonYu1990}; see also \cite{BrownGiustoSimpson2002}]
Over $\RCA$, $\WWKL$ is equivalent to the statement that for every 
compact separable metric space $X$ and every measure $\mu$ on 
$X$, $\mu$ is countably additive.  That is, for every sequence of 
open sets $U_n$,
$$\lim_N \mu(\cup_{n<N} U_n) = \mu(\bigcup_n U_n).$$
\end{theorem}

\begin{corollary}[$\WWKL$]
For all open sets $U \subseteq 2^\omega$, $\mu(U) = \mu_I(U)$.
\end{corollary}

One final intensional notion of a 
measurable set is 
needed for the development of measure theory.  

\begin{definition}
A \emph{rapidly null} $G_\delta$ set is a $G_\delta$ set $\bigcap_n U_n$ such 
that for each $n$, $\mu_I(U_n) < 2^{-n}$.
\end{definition}

Note: a Martin-L\"of test is just a computably presented rapidly null $G_\delta$ set.

\begin{theorem}[\cite{AvigadDeanRute2012}]\label{thm:adr}
Over $\RCA$, $\WWKL$ is equivalent to the statement that 
if $A$ is a rapidly null $G_\delta$ subset of $2^\omega$,
then $A \neq 2^\omega$.
\end{theorem}

Thus in $\WWKL$, a $\mu$-measurable set may be non-vacuously 
defined as follows.  Let $\mu:C(X) \rightarrow \mathbb R$ be 
a positive Borel probability measure.  Let $L^1(X,\mu)$ denote 
the completion of $C(X)$ with respect to the $L^1$ norm defined by
$||f-g||_1 = \int |f-g|$.
Recall that a sequence $\langle x_n\rangle$ of points of a metric space 
is called \emph{rapidly Cauchy} if for all $n$, we have 
$d(x_n,x_{n+1}) < 2^{-n}$.
Each element of $L^1(X,\mu)$ is represented by 
many \emph{names}, where a name is a sequence 
$\langle f_n\rangle_{n\in\omega}$ of functions from $C(X)$ 
that is rapidly Cauchy for the $L^1$ norm.

\begin{definition}[\cite{BrownGiustoSimpson2002}]\label{defn:measurable_set}
A \emph{measurable characteristic function} is a function $f \in L^1(X,\mu)$
such that $f(x) \in \{0,1\}$ for all $x$ outside a rapidly 
null $G_\delta$ set.  A set $E$ is \emph{measurable} if there is 
some $f \in L^1(X,\mu)$ such that $f = \chi_E$ outside 
a rapidly null $G_\delta$ set.
\end{definition}
Here $\chi_E$ denotes the characteristic function of $E$.  
The measure of $E$ is then defined as $\mu(E) = \mu(f)$,
where $f = \chi_E$ almost everywhere as above.  This is well-defined and 
locally well-behaved by the 
following results of X. Yu \cite{YuX1994}.
\begin{theorem}[$\WWKL$]\label{Thm2.8}
For $f,f' \in L^1(X,\mu)$, $||f-f'||_1 = 0$ if and only if $f = f'$
outside of a rapidly null $G_\delta$ set.  If $f \leq f'$ outside of
a rapidly null $G_\delta$ set, then $\mu(f) \leq \mu(f')$.
\end{theorem}
For the rest of this paragraph, $\WWKL$ is assumed.
Observe now that if $U\subseteq 2^\omega$ is open and if $U$ is measurable
in the above sense (that is, $\chi_U \in L^1(X,\mu)$), 
then we have $\mu(U) = \mu_I(U) = \mu(\chi_U)$.
The last equality follows because if $U = \bigcup_{i<\omega} [p_i]$,
the functions $\chi_{\cup_{i<n} [p_i]}$ are continuous and 
converge 
to $\chi_U$ in the $L^1$ norm.  Finally, if $A$ is a rapidly 
null $G_\delta$ set, then $\mu(\chi_A) = \mu_I(A) = 0$ 
because $\chi_A = 0$ outside of $A$ itself.
Therefore, when measurable characteristic 
functions for open or rapidly null $G_\delta$ sets exist, all our
ways of defining measures for these sets coincide.  
The existence of a measurable characteristic function for an open set also 
guarantees that the measure of that open set exists in the 
model (and thus can be discussed directly, not just comparatively).

Finally, we will need to make use of some more explicit versions 
of known results from the literature.  For example, we want to 
use Theorem \ref{Thm2.8}, but as stated it does not give any  
bounds on the complexity of the rapidly null $G_\delta$ set. 
However, those bounds do exist and 
we need the uniformity that comes with them.  So below we 
reprove several results in order to 
clarify the complexity of the null set of points that are being
discarded.  From here forward, we also restrict our attention to 
Cantor space with the fair coin measure, which is denoted by $\lambda$.

First, recall that if $A_n$ is a sequence of rapidly null $G_\delta$ 
sets $A_n = \bigcap_{i} A_{n,i}$, the same trick used for producing 
a universal Martin-L\"of test can also produce a rapidly null $G_\delta$ 
set $A \supseteq \bigcup_n A_n$.  Just let $U_j = \bigcup_n A_{n,n+j+1}$, 
and let $A = \bigcap_j U_j$.

Much but not all of the rest of this section has been presented  
in \cite{BrownGiustoSimpson2002}.
\begin{proposition}[$\WWKL$]\label{Prop2.10}
Suppose that $\langle f_i\rangle$ is a sequence of ideal 
continuous functions of $C(X)$ which is rapidly Cauchy 
for the $L^1$ norm.  Let 
$$A_n = \{ x : \exists N \sum_{i=2n+1}^N |f_i(x)-f_{i+1}(x)| > 2^{-n}\}$$
Then $\mu(A_n) \leq 2^{-n}$.
\end{proposition}
\begin{proof}
Formally, $A_n$ is a union of basic open sets $\bigcup_j [p_j]$ 
satisfying the condition.  We can assume the $[p_j]$ 
are disjoint.  By countable additivity, 
it suffices to show that $\mu(B) < 2^{-n}$
for all sets $B = \cup_{j<k} [p_j]$.  Let $N$ be large 
enough to witness that $[p_j] \subseteq A_n$ for all $j<k$.
We have 
$$2^{-n}\mu(B) = \int 2^{-n}\chi_B \leq \int \sum_{i=2n+1}^N |f_i-f_{i+1}|
= \sum_{i=2n+1}^N \int |f_i-f_{i+1}| < 2^{-2n}$$
Thus $\mu(B) < 2^{-n}$, as needed.
\end{proof}

The corollaries use $\ACA$ only to guarantee that a Cauchy sequence 
converges.
\begin{corollary}[$\ACA$]\label{Cor2.11}
A name $\langle f_i\rangle$ for an element of $L^1(2^\omega)$
converges pointwise a.e.  Furthermore, this pointwise convergence is 
achieved outside of the rapidly null $G_\delta$ set
$$\bigcap_k \bigcup_{n>k} A_n$$
where $A_n$ are defined as above.
\end{corollary}

\begin{corollary}[$\WWKL$]\label{cor:extra}
A name $\langle f_i\rangle$ for an element of $L^1(2^\omega)$ 
converges uniformly on each closed set
$$B_k = 2^\omega \setminus \bigcup_{n\geq k} A_n,$$
where $A_n$ are defined as above.
Furthermore, the modulus of uniform convergence of 
$f_i$ on $B_k$ is primitive recursive: if $m>2\max\{\ell,k\}$,
then $|f_m(x) - f(x)| \leq 2^{-\ell}$.
\end{corollary}
\begin{proof} Let $n = \max\{\ell,k\}$ and $x\in B_k$.
Then $B_k \cap A_n = \emptyset$, 
and thus the series 
$f_m(x) + \sum_{i=m}^\infty (f_{i+1}(x)-f_i(x))$ converges 
absolutely, with
$$\sum_{i=m}^\infty |f_{i+1}(x)-f_i(x)|  \leq 
\sum_{i=2n+1}^\infty |f_{i+1}(x)-f_i(x)| \leq 2^{-n} \leq 2^{-\ell}.$$
\end{proof}

\begin{corollary}[$\ACA$]\label{Cor2.12}
If $\langle f_i\rangle$ and $\langle g_i\rangle$ are two names for the same 
element of $L^1(2^\omega)$, then 
$$\lim_i f_i(x)  = \lim_i g_i(x)$$
for almost all $x$.  Furthermore, this pointwise convergence is 
achieved outside of a rapidly null $G_\delta$ set given by an 
explicit formula.
\end{corollary}
\begin{proof}
Let $A_n(f), A_n(g),$ and $A_n(f,g)$ be defined as in Proposition 
\ref{Prop2.10} applied to the rapidly Cauchy sequences
$\langle f_i\rangle$, $\langle g_i \rangle$,
and $\langle f_2,g_3,f_4,g_5,\dots \rangle$ respectively.  Then 
the limits of $f_i(x)$ and $g_i(x)$ exist and agree for any $x$ outside of
three rapidly null $G_\delta$ sets.  Combine these rapidly null $G_\delta$ 
sets into a single rapidly null $G_\delta$ set.
\end{proof}

\begin{proposition}[$\ACA$]\label{Prop2.13}
If $\langle h_j\rangle$ is a sequence of functions of $L^1(2^\omega)$ 
rapidly converging to a function $g \in L^1(2^\omega)$, then
$$\lim_{j\rightarrow \infty} h_j(x) = g(x)$$
for almost all $x$.  Furthermore, this pointwise convergence is 
achieved outside of a rapidly null $G_\delta$ set given by 
an explicit formula.
\end{proposition}
\begin{proof}
Define $\langle f^i\rangle_{i\in \omega}$ by 
$f^i = h_i^{2i+1}$, where $\langle h_j^i\rangle_{i<\omega}$ is the given name 
for $h_j$.  Then $\langle f^{i+2} \rangle_{i \in \omega}$ is rapidly Cauchy and is
another name for $g$, which we can 
see because
$$\int |f^i - f^{i+1}| \leq \int |f^i - h_i| + \int |h_i -h_{i+1}| + \int |h_{i+1}-f^{i+1}| \leq 2^{-2i} + 2^{-i} + 2^{-2i}$$
and
$$\int |f^i-g| \leq \int |f^i - h_i| + \int |h_i - g| \leq 2^{-2i} + 2^{-i+1}.$$

Let $A_n(g)$ and $A_n(h_j)$ be the building blocks of infinitely many 
rapidly null $G_\delta$ sets as in Corollary \ref{Cor2.11},
so that outside of these sets the notations $g(x)$ and $h_j(x)$ 
are well-defined as the pointwise limits of the given names for $g$ and 
each $h_j$.  Additionally,
letting 
$$C_k = \bigcup_{\substack{j>k\\ n>j}}A_n(h_j),$$ 
by Proposition \ref{Prop2.10}, we have $\lambda(\bigcup_{n>j} A_n(h_j)) < 2^{-j}$ 
and thus 
$\lambda(C_k) < 2^{-k}$ 
and 
$\bigcap_k C_k$ is a
rapidly null $G_\delta$ set.  Combine into a single test
\begin{enumerate}
\item\label{e1} the infinitely many rapidly null $G_\delta$ sets which result from applying 
Corollary \ref{Cor2.11} to the given names for $g$ and each $h_j$
\item\label{e2} the rapidly null $G_\delta$ set guaranteed by Corollary \ref{Cor2.12},
so that for $x$ outside of $B$, $\lim_i f^i(x) = g(x)$.
\item\label{e3} $\bigcap_k C_k$.
\end{enumerate}  
By (\ref{e1}), if $x$ avoids this test, then $h_j(x)$ and $g(x)$ are well-defined 
as the pointwise limit of the given names of $g$ and $h_j$.
By (\ref{e2}), if $x$ avoids this test, then $\lim_i f^i(x) = g(x)$.
Finally, we claim that if $x$ avoids this test, then 
$\lim_j h_j(x) = \lim_i f^i(x)$.  The limit on the right hand side exists, 
so it suffices to show that $\lim_j |f^j(x) -h_j(x)| = 0$.
This follows by (\ref{e3}) because 
if $x \notin C_k$ for some $k$, then for all $j>k$ we have 
$$|f^j(x) - h_j(x)| \leq \sum_{i=2j+1}^\infty |h_{j}^i(x) - h_{j}^{i+1}(x)| \leq 2^{-j}.$$
\end{proof}

We have the following relationship between higher randomness and measure
theory.  This is surely known (and one could surely do better than $\Delta^1_1$-random)
but it is enough for our purposes.
\begin{lemma}\label{Lem1}
Suppose that $f \in L^1(2^\omega)$, with name $\langle f^i\rangle_{i<\omega}$. 
Suppose that $R$ is $\Delta^1_1$-random relative to $\langle f^i \rangle_{i<\omega}$.
Then 
$$\lim_{N\rightarrow \infty} \frac{1}{N} \sum_{j<N} f(R^{[j]}) = \int_{2^\omega} f$$
\end{lemma}
\begin{proof}
Note that the randomness of $R$ ensures that $f(R^{[j]})$ is well-defined as 
$\lim_i f^i(R^{[j]})$.
For any $\varepsilon$, we can find a measurable function
$f_\varepsilon = \sum_{k=-\infty}^\infty k\varepsilon \chi_{A_k}$
where $A_k$ are measurable sets which have Borel definitions uniformly 
in the name $\langle f^i\rangle$, and such that 
$|f(x) - f_\varepsilon(x)|<\varepsilon$ for all $x$ outside of a $G_\delta$ 
set which also has a Borel definition relative to $\langle f^i\rangle$.  
Then the randomness of $R$ ensures that the $R^{[j]}$ visit each 
$A_k$ with the right limiting frequency, and that 
$|f(R^{[j]}) - k\varepsilon| < \varepsilon$
whenever $R^{[j]}\in A_k$.  Thus $\frac{1}{N} \sum_{j<N} f(R^{[j]})$ 
is within $\varepsilon$ of $\frac{1}{N} \sum_{j<N} f_\varepsilon(R^{[j]})$,
and the latter tends to to $\int_{2^\omega} f_\varepsilon$ as $N$ increases.  
Letting $\varepsilon$
go to zero completes the proof.
\end{proof}

\section{Regularity approximations and measure approximations}

The following version of measurability for a
set was implicit in \cite{YuX1993}.
\begin{definition}\label{defn:regularity_measurable_set}
A set $B$ is \emph{regularity-measurable} if there are $G_\delta$ 
sets $A$ and $C$ such that $A^c \subseteq B \subseteq C$ 
and $A \cap C$ is rapidly null.
\end{definition}
We bring up this definition because such a pair $(A,C)$,
which we could call a \emph{regularity approximation} to 
$B$, would seem an obvious analog to the 
\emph{Baire approximation} to a set $B$ defined in \cite{ADMSW}.
We can use this notion of measurability to define the principle 
$\CDM$ as follows.
\begin{definition}
Let $\CDM$ be the principle ``Every completely determined 
Borel set is regularity-measurable''.
\end{definition}
A difference between measure and category now arises.
The Baire Category Theorem holds in $\RCA$, so $\RCA$ 
knows that the whole space is not meager.  However, 
$\WWKL$ is needed in order to know that the whole space 
is not null.
\begin{proposition}\label{meetRCA}
Over $\RCA$, $\neg \WWKL$ implies $\CDM$.
\end{proposition}
\begin{proof}
By Theorem \ref{thm:adr}, let $A$ be an rapidly null $G_\delta$ 
set with empty complement.  Let $C=2^\omega$.  Then for any set
$B$, we have $A^c = \emptyset \subseteq B \subseteq C$, 
but $A \cap C$ is rapidly null because $A$ is rapidly null.
\end{proof}
In the presence of $\WWKL$, however, regularity-measurable 
coincides with the same notion of measurability 
given in Definition \ref{defn:measurable_set}.
\begin{proposition}[$\WWKL$] Let $B\subseteq 2^\omega$ be 
any set.  (Formally, the membership of $B$ can be given 
by any formula in the language of second order arithmetic).  Then 
$B$ is regularity-measurable if and only if it is measurable 
in the sense of Definition \ref{defn:measurable_set}.\end{proposition}
\begin{proof}
Suppose $B$ is regularity-measurable.  It follows that $A \cup C = 2^\omega$.
Therefore, if $A = \bigcap_n A_n$ and $C = \bigcap_n C_n$, we have for each
$n$ that $A_n \cup C_n = 2^\omega$.  Using $\WWKL$, it follows that 
$\mu(A_n \cup C_n) = 1$, while $\mu(A_n \cap C_n) < 2^{-n}$ because 
$A \cap C$ is rapidly null.  Define a sequence of functions $f_n:2^\omega\rightarrow \{0,1\}$ and open sets $B_n$ as follows.  Given $n$, let $s$ be large enough that $\mu(D_{n+1,s}) < 2^{-(n+1)}$, where we define
$$D_{n,s} = 2^\omega \setminus (A_{n,s} \cup C_{n,s}).$$  Let $f_n$ be the 
characteristic function of $C_{n+1,s}$, and let 
$$B_n = (A_{n+1} \cap C_{n+1}) \cup D_{n+1,s}.$$
Then $\mu(B_n) < 2^{-n}$.  
We have $$||f_n - f_m||_1 = \mu(A_{n+1,s} \Delta A_{m+1,t})$$
where $s$ and $t$ are chosen as in the definition.  Since 
$A_{n+1,s} \Delta A_{m+1,t} \subseteq B_n \cup B_m$, 
the sequence $\langle f_n\rangle$ is rapidly Cauchy 
and $f_n(x)$ converges to $\chi_B(x)$ for all $x$ outside of $\bigcap_n (\bigcup_{k>n} B_k)$.

On the other hand, if $B$ is measurable in the sense of Definition 
\ref{defn:measurable_set}, then if $\langle f_n\rangle_{n\in\omega}$ 
is an $L^1$-name for $\chi_B$,
the sets $A_n = \{x : f_n(x) < 2/3\}$ and $C_n = \{x : f_n(x) > 1/3\}$
demonstrate that $B$ is regularity-measurable.  This follows because, 
letting $D = A_n\cap C_n$, we have
$$\frac{1}{3}\mu(D) = \int_D \frac{1}{3} \leq  \int_D |f_n - \chi_B| \leq ||f_n-\chi_B||_1 \leq 2^{-n+1}.$$
\end{proof}

The first step in evaluating the strength of $\CDM + \WWKL$ is immediate.
\begin{proposition}\label{joinACA}
Over $\WWKL$, the statement ``Every open subset of $2^\omega$ 
is measurable'' is equivalent to $\ACA$.
\end{proposition}
\begin{proof}
It is clear that $\ACA$ proves the given statement.  In the other 
direction, given an increasing sequence of real numbers $\langle a_n\rangle$ 
with each $a_n < 1$, let $U$ be an open set designed so that 
$\mu_I(U) = \sup_n a_n$.  For example, let $U$ be the 
set which contains exactly 
those cylinders $[p\concat 0]$ such that for some $n$, we have $.p\concat1<a_n$, 
where $.p\concat 1$ denotes the rational number with binary decimal 
expansion given by $p\concat1$.  By $\WWKL$, $\mu_I(U) = \mu(U)$.
But $\mu(U)$ exists as a number, thus $\sup_n a_n$ exists.
\end{proof}

Combining Propositions \ref{meetRCA} and \ref{joinACA}, we arrive 
at the following curiosity.  Let $\mathsf{OSM}$ be the statement 
``Every open set is regularity-measurable''.  Then by Proposition 
\ref{joinACA}, we have that $\ACA$ is equivalent to $\WWKL + \mathsf{OSM}$,
while Proposition \ref{meetRCA} shows that 
$\RCA$ proves $\WWKL \vee \mathsf{OSM}$ (here $\vee$ denotes a disjunction 
of two principles, not a 
a join operator on those principles).  Thus we have a diamond 
formed of reasonably natural principles, though it must be admitted that 
$\mathsf{OSM}$ does not mean much outside of $\WWKL$.  We 
are not aware of any other diamond in reverse mathematics.  By a diamond 
here we just mean informally an incomparable pair of principles $A$ and $B$ such that 
$A + B$ is equivalent to some principle of interest, while 
$A \vee B$ follows from $\RCA$.

We return now to our main discussion of the principle $\CDM$.  
One direction of Proposition \ref{joinACA} can be extended to 
the Borel case as follows.

\begin{proposition}\label{prop:cdm-implies-lwca}
Over $\WWKL$, $\CDM$ implies $\LwCA$.
\end{proposition}
\begin{proof}
Given $\langle \phi_n \rangle_{n\in \omega}$ a sequence of completely determined 
formulas of $L_{\omega_1,\omega}$, turn them into Borel codes by 
change $\bigcap$ to $\bigwedge$, $\bigcup$ to $\bigvee$, and
changing their leaves as follows.  If $\phi_n$ has {\tt true} at a leaf, 
replace it with $[0^n1]$.  If $\phi_n$ has {\tt false} at a leaf, replace it with 
$\emptyset$.  Now take the union of all of these codes.  The resulting 
code is completely determined because each $\phi_n$ was completely 
determined and each $X \in2^\omega$ belongs to at most one cylinder $[0^n1]$.
If $f$ is a measurable characteristic function, then $f$ is almost surely 1 
on $[0^n1]$ whenever $\phi_n$ is true, and almost surely 0 on 
$[0^n1]$ whenever $\phi_n$ is false.  Thus the sequence 
$\langle 2^n \int_{[0^n1]} f\rangle_{n\in \omega}$
witnesses the satisfaction of $\LwCA$; this sequence 
assigns 1 to the true formulas and 0 to the false ones.
\end{proof}

The classical way of showing that every Borel set is measurable
is to use arithmetic transfinite recursion to 
define a regularity approximation to $\makeset{T_\sigma}$
for each $\sigma \in T$.  We present
an effectivization of the 
classical proof which is particularly well-suited to our subsequent 
analysis.

\begin{definition}
Let $T$ be a code for a Borel set.  A \emph{measure decomposition}
for $T$ is a collection $\langle f_\sigma : \sigma \in T\rangle$, where 
each  $f_\sigma \in L^1(2^\omega)$,
such that
\begin{enumerate}
\item If $\sigma$ is a leaf, then $f_\sigma$ is the characteristic 
function of $|T_\sigma|$.
\item If $\sigma$ is a union, then $f_\sigma = \sup_n f_{\sigma n}$.
\item If $\sigma$ is an intersection, then $f_\sigma = \inf_n f_{\sigma n}$.
\end{enumerate}
\end{definition}
All three equalities above refer to equality in the sense of the metric space $L^1(2^\omega)$.  For example, the equation $f_\sigma = \sup_n f_{\sigma n}$ is 
shorthand for
$$\lim_{N \rightarrow \infty} \left(\sup_{n<N} f_{\sigma n}\right) = f_\sigma$$
and similarly for the other equation.  In all cases, $n$ ranges only over those 
numbers for which $\sigma n \in T$.

\begin{proposition}[$\ACA$]\label{Prop3.7}
Suppose $T$ is a code for a completely determined Borel set.  If 
$T$ has a measure decomposition, then $|T|$ is measurable.
\end{proposition}
\begin{proof}
We need to show that $f_{\emptyset}$ is a.e. equal to the 
characteristic function of $|T|$.  This is proved by arithmetic transfinite 
induction on $T$.

Observe that if we were willing to use $\Sigma^1_2$ transfinite induction and 
$\Sigma^1_1$-$\mathsf{AC}$, 
the proof which inducts on the following statement would be very short:
there is a rapidly null $G_\delta$ such that for all $X$ outside of it, 
$f_\sigma(X) = 1$ if and only if $X \in |T_\sigma|$.  Since we want to get away 
with arithmetic transfinite induction only, 
we need to identify the rapidly null $G_\delta$ 
in advance, then fix some $X$ outside it, and then prove $f_\emptyset(X)$ 
is correct by transfinite induction on $T$.

We claim the following collection of
rapidly null $G_\delta$ sets exists:
\begin{enumerate}
\item\label{gdelta1} For all $\sigma$, a rapidly null $G_\delta$ such that for all 
$x$ outside of it, the name of $f$ converges at $x$. 
\item\label{gdelta2} For all leaf $\sigma$, a rapidly null $G_\delta$ set such 
that on its complement, $f_\sigma$ is the characteristic function of $|T_\sigma|$ 
\item\label{gdelta3} For all union $\sigma$, a rapidly null $G_\delta$ set 
such that for all $x$ in its complement, $f_\sigma(x) = \sup_n f_{\sigma n}(x)$
\item\label{gdelta4} For all intersection $\sigma$, same as the above except using 
$\inf_n f_{\sigma n}$.
\end{enumerate}
The sets in (\ref{gdelta1}) are obtained by uniform application of Corollary \ref{Cor2.11} 
to the given names for the functions $f_\sigma$.  The sets in (\ref{gdelta2}) are obtained 
by uniform application of Corollary \ref{Cor2.12} to $f_\sigma$ and a standard name for 
the characteristic function of the clopen set $|T_\sigma|$.  To obtain (\ref{gdelta3}),
use the fact that 
$$\lim_{N \rightarrow \infty} \left(\sup_{n<N} f_{\sigma n}\right) = f_\sigma,$$
define $h_N = \sup_{n<N} f_{\sigma n}$, and
find a sequence $N_i$ such that 
$\langle h_{N_i} \rangle_{i\in\omega}$
is rapidly convergent to $f_\sigma$.  Then apply Proposition \ref{Prop2.13} to 
$\langle h_{N_i}\rangle_{i\in\omega}$ together 
with the given name for $f_\sigma$.  Although we have passed to a subsequence, 
because $h_N(x) \leq h_{N+1}(x)$ for all $x$, it follows that $h_N(x)$ converges 
if and only if $h_{N_i}(x)$ converges.  (It will happen in our situation that 
$h_N(x)$ converges for all $x$, though we do not need this.)
The procedure for (\ref{gdelta4}) is similar.

Let $A$ be a rapidly null $G_\delta$ set which 
contains all the bad-behavior sets above.  Fix $X\not\in A$.  We claim that 
the map which sends $\sigma$ to $f_\sigma(X)$ is an 
evaluation map for $X$ in $T$.  That is, we claim $f_\sigma(X) = 1$ 
if and only if $X \in |T_\sigma|$.
The claim is proved by arithmetic transfinite induction on $T$. 
Observe that $A$ contains all the points at which 
the proposed evaluation map fails to be right at the leaves or 
fails to satisfy the logic of the tree.

In particular, $f_\emptyset(X) = 1$ if and only if $X \in |T|$.
\end{proof}

Uniformly arithmetic in a sequence $\langle f_{\sigma n}\rangle_{n\in\omega}$, we may 
produce the functions $\sup_n f_{\sigma n}$ and 
$\inf_n f_{\sigma n}$.  Therefore, $\ATR$ suffices to 
create measure decompositions for all Borel sets.  However, 
$\ACA$ is enough to guarantee their uniqueness.

\begin{proposition}[$\ACA$]
Suppose that $T$ is a Borel code and 
$\langle f_\sigma \rangle_{\sigma \in T}$ and 
$\langle g_\sigma \rangle_{\sigma \in T}$ are two measure decompositions 
for $T$.  Then for all $\sigma \in T$, $f_\sigma = g_\sigma$ as $L^1$ functions.
\end{proposition}
\begin{proof}
By arithmetic transfinite induction.  If for all $n$, $f_{\sigma n} = g_{\sigma n}$, 
then for all $N$, $\sup_{n<N} f_{\sigma n} = \sup_{n<N} g_{\sigma n}$.  
Therefore, these sequences have the same limit in the sense of $L^1$.
\end{proof}

Although we will show in the next section that $\WWKL+\CDM$
is strictly weaker than $\ATR$, the existence of measure
decompositions is still necessary for $\WWKL+\CDM$ to hold.
Therefore, any model of $\WWKL+\CDM + \neg \ATR$ will 
need some other way of producing measure decompositions.

\begin{proposition}[$\ACA$]
If $\CDM$ holds,
then every completely determined Borel set has a measure
decomposition.
\end{proposition}
\begin{proof}
For any Borel 
code $S$, define an operation $S[n]$ as follows.  Whenever a leaf of 
$S$ is labeled by the clopen set $[p_0]\cup \dots \cup [p_k]$, 
replace it with the clopen set $[0^n1p_0]\cup \dots \cup [0^n1p_k]$.  This 
has the effect of shrinking the set coded by $S$ and relocating 
it to live completely inside the cone $[0^n1]$.

Let $h:\omega\rightarrow T$ be a computable surjection.  
If $T$ is completely determined, so is $\tilde T$, where 
$$\tilde T = \bigcup_{n \in \omega} T_{h(n)}[n]$$
Colloquially, $\tilde T$ has been formed by taking each 
subtree $T_\sigma$ of $T$ and giving it its own dedicated part of the 
Cantor space.  Now, if $\tilde T$ is measurable via the function $f\in L^1$,
then the functions
$$f_{\sigma}(X) = f(0^{\min h^{-1}(\sigma)}1\concat X)$$
are a measure decomposition for $T$.
\end{proof}

\section{Results}

In this section we construct an $\omega$-model $\mathcal M$ 
which satisfies $\CDM$ but not $\ATR$.
Let $R$ be a $\Pi^1_1$-random.  Let $\mathcal M$ 
be the $\omega$-model whose second-order part is $\bigcup_{i<\omega} \Delta^1_1(\bigoplus_{k<i}R^{[k]})$, where $R^{[k]}$ denotes the $k$th column of $R$.  

Since the 
strings of $2^{<\omega}$ are in one-to-one correspondence with $\omega$, 
we can assume such a correspondence is fixed and abuse notation 
to also let $G^{[p]}$ denote a column of $G$ whenever $p \in 2^{<\omega}$
 and $G \in 2^\omega$.

\begin{proposition}
The model $\mathcal M$ does not satisfy $\ATR$.
\end{proposition}
\begin{proof}
Let $a^\ast$ be a computable pseudo-ordinal.  Then $a^\ast \in \mathcal M$.  We claim 
that $a^\ast$ has neither a descending sequence, nor a jump hierarchy, 
in $\mathcal M$.  If $\Delta^1_1(R_0)$ had one, where $R_0 = \bigoplus_{k<i} R^{[k]}$,
then by Theorem \ref{thm:nomoveomega1},
$\omega_1^{R_0} = \omega_1^{ck}$.  Thus there is an ordinal $b \in \kO$ 
such that $H^{R_0}_b$ computes either a jump hierarchy on or a descending sequence in 
$a^\ast$.  But recognizing a jump hierarchy or a descending sequence is 
arithmetic.  So
$$\text{}``H^X_b \text{ computes a jump hierarchy or descending sequence for } a^\ast\text{''}$$
is a $\Sigma^0_{b + O(1)}$ statement, and it has measure either 0 or 1 because 
it describes a property of the tail of $X$.  Because $R_0$ is sufficiently random, 
and satisfies the statement, the set has measure 1.  But then any $b+O(1)$-generic 
also satisfies the statement.  This is a contradiction because there are 
$b+O(1)$-generics in $HYP$, but $a^\ast$ has no hyperarithmetic descending 
sequence nor any hyperarithmetic jump hierarchy.
\end{proof}

\begin{proposition}\label{prop:lwca}
The model $\mathcal M$ satisfies $\LwCA$.  Furthermore, whenever 
$R_0 \in M$ and $\langle \phi_i \rangle \in \Delta^1_1(R_0)$, 
if $\langle \phi_i \rangle$ is completely determined in $\mathcal M$, 
then it is completely determined in $\Delta^1_1(R_0)$.
\end{proposition}
\begin{proof}
Suppose that $\langle \phi_j \rangle \in \Delta^1_1(\bigoplus_{i<k} R^{[i]})$ is a sequence of 
formulas of $\LwCA$ which is completely determined in $\mathcal M$. 
Since $\LwCA$ is a theory of hyperarithmetic analysis, it suffices to 
show that the sequence is determined in $\Delta^1_1(R_0)$,
where $R_0 = \bigoplus_{i<k} R^{[i]}$.  Fixing $j$, there is an $m>k$ such that 
$\Delta^1_1(\bigoplus_{i<m} R^{[i]})$ contains an evaluation map for $\phi_j$.
Let $R_1 = \bigoplus_{k\leq i<m} R^{[i]}$.  By Van Lambalgen's Theorem
for $\Pi^1_1$-randoms, $R_0$ and $R_1$ are relatively $\Pi^1_1$-random. 
Since $\omega_1^{R_0\oplus R_1} = \omega_1^{ck}$, there is some $a \in \kO$
such that this evaluation map is computable from $H_a^{R_0\oplus R_1}$.  Then
$$C_j := \{ X : H_a^{R_0\oplus X} \text{ computes an evaluation map for } \phi_j\}$$
is a $\Delta^1_1(R_0)$ set which contains the $\Pi^1_1(R_0)$-random $R_1$.  
Therefore, $C_j$ has measure 1, so any sufficiently random element computes 
an evaluation map for $\phi_j$.  Here, sufficiently random just means more 
random (relative to $R_0$) than the descriptive complexity of $C_j$.  So there 
are elements of $\Delta^1_1(R_0)$ that are sufficiently random.  Thus $\phi_j$ 
is determined in $\Delta^1_1(R_0)$.
\end{proof}

To show that $\mathcal M$ models $\CDM$,
the following classical fact will be useful.  It says roughly 
that if you approximate a bounded 
function $f$ by using its average values on smaller and smaller partitions of the domain,
the resulting sequence converges to $f$ in the $L^1$ sense.
\begin{lemma}[$\WWKL$]\label{Lem2}
If $f \in L^1(2^\omega)$ is bounded
and $h_i = \sum_{p \in 2^i} (2^i \int_{[p]} f)\chi_{[p]}$, 
then $h_i \rightarrow f$ in the $L^1$ norm.  
\end{lemma}
\begin{proof}
Given $\varepsilon$, use Corollary \ref{cor:extra} to find a closed set $B$ 
such that 
the restriction of $f$ to 
$B$ is continuous, and $\mu(2^\omega)-B < \varepsilon/M$, where $M$ is a 
bound on $f$.  Let $i$ be large enough that on $B$, if $x\uhr i = y\uhr i$, 
then $|f(x) - f(y)|<\varepsilon$.  Then for all strings $p\in2^i$ and all $x_0\in [p]$,
\begin{align*}
|h_i(x) - f(x)| &=  {\textstyle | (2^i\int_{[p]} f) - f(x_0)|}\\
&= {\textstyle |(2^i\int_{[p]\cap B} f) + (2^i\int_{[p]\setminus B} f) - 2^i \int_{[p]} f(x_0) |} \text{\ \ ($f(x_0)$ is a constant.)}\\
&{\textstyle \leq |2^i \int_{[p]\cap B} (f-f(x_0))| + |2^i \int_{[p]\setminus B} f| + |2^i \int_{[p]\setminus B}f(x_0)|}\\
&\leq {\textstyle 2^i \int_{[p]\cap B} \varepsilon + 2(2^i \int_{[p]\setminus B} M)}
\end{align*}
Therefore,
\begin{align*}
\int |h_i -f| &= \sum_{p \in 2^i} \int_{[p]} |h_i-f|\\
&\leq \sum_{p \in 2^i} \int_{[p]} 2^i({\textstyle \int_{[p\cap B]} \varepsilon + 2\int_{[p]\setminus B} M})\\
&\leq \sum_{p\in 2^i} ({\textstyle \int_{[p\cap B]} \varepsilon + 2\int_{[p]\setminus B} M})\\
&= {\textstyle \int_{B}\varepsilon + 2 \int_{2^\omega\setminus B} M \leq \varepsilon + 2 \varepsilon.}
\end{align*}
\end{proof}

\begin{lemma}\label{Lem4.2}
Suppose that $f \in L^1(2^\omega)$, with name $\langle f^i\rangle_{i<\omega}$. 
Suppose that $R$ is $\Delta^1_1$-random relative to $\langle f^i \rangle_{i<\omega}$.
Define a sequence of functions $g^i$ by
$$g^i(X) = \lim_{N\rightarrow \infty} \frac{1}{N}\sum_{j<N} f((X\uhr i)\concat R^{[j]})$$
Then the functions are well-defined and $g^i \rightarrow f$ in the $L^1$ norm.
\end{lemma}
\begin{proof}
By $2^i$-many applications of Lemma \ref{Lem1} to the functions 
$f_p(R):= f(p\concat R)$, and since    $2^i \int_{[p]} f = \int_{2^\omega} f_p$,
we have $g^i(X) = \sum_{p \in 2^i} (2^i \int_{[p]} f)\chi_{[p]}$.  Then $g^i\rightarrow f$ by Lemma 
\ref{Lem2}.
\end{proof}

\begin{theorem}
Over $\WWKL$, $\CDM$ is strictly weaker than $\ATR$.  In particular,
$\mathcal M$ satisfies $\WWKL+\CDM$ but not $\ATR$.
\end{theorem}
\begin{proof}
Suppose that we are given $T$, a completely determined Borel code. 
To simplify notation, we assume that $T \in \Delta^1_1$; the result for 
arbitrary $T \in \mathcal M$ follows by relativization.  Let $R_0 = R^{[0]}$.  Then
abusing the column notation further, consider $R_0$ as being 
made out of infinitely many distinct and computably identifiable columns,
one column for each pair $(\sigma, j)$, where $\sigma \in \omega^{<\omega}, j \in \omega$, and let $R^{[\sigma,j]}$ denote the column
allocated to that pair.  Then letting
$$U:= \{(p,\sigma,j) : p\concat R_0^{[\sigma,j]} \in |T_\sigma|\}$$
we have $U \in \Delta^1_1(R_0)$ by Proposition \ref{prop:lwca}.  
By the same reasoning, we also have that 
$U_\sigma := \{(p,j) : (\sigma, p,j) \in U\}$ satisfies $U_\sigma \in \Delta^1_1(R_0^{[\sigma]})$,
where $R_0^{[\sigma]} = \bigoplus_{j<\omega} R_0^{[\sigma,j]}$.

Therefore, in $\Delta^1_1(R_0)$ we can also find the array of functions
$\langle f_\sigma^i\rangle_{\sigma \in T,i\in\omega}$ defined as follows.  
$$f_\sigma^i(X) := \limsup_{N\rightarrow\infty} \frac{1}{N} \sum_{j<N} U_\sigma(X\uhr i,j)$$
Then define $f_\sigma = \limsup_i f_\sigma^i$.
Since the functions $X\mapsto U_\sigma(X\uhr i, j)$ are continuous 
and bounded above by 1, $f_\sigma \in L^1(2^\omega)$ by the monotone 
convergence theorem.
Of course, the intention is to show that all limsups above can be replaced by 
limits a.e., and that $f_\sigma$
represents $|T_\sigma|$ as a measurable set.  We prove that $\langle f_\sigma\rangle$ 
is a measure decomposition 
by arithmetic transfinite induction within $\mathcal M$.

If $\sigma$ is a leaf then the sequence of functions $f_\sigma^i$ is eventually 
constant and equal to the characteristic function of the clopen set coded by
$T_\sigma$, as desired.

So to complete the proof that $\langle f_\sigma \rangle_{\sigma \in T}$ 
is measure decomposition, it suffices to show that $\mathcal M$ models
the following statement for each non-leaf $\sigma \in T$:

``If for all $n$, $\langle f_{\sigma n  \tau}\rangle_{\tau \in T_{\sigma n}}$ is a 
measure decomposition, then $\langle f_{\sigma\tau}\rangle_{\tau \in T_\sigma}$ 
is a measure decomposition.''
That is, assuming $\mathcal M$ models
the hypothesis, we need to show that $\mathcal M$ models:
\begin{enumerate}
\item If $\sigma$ is a union, then $f_\sigma = \sup_n f_{\sigma n}$
\item If $\sigma$ is an intersection, then $f_\sigma = \inf_n f_{\sigma n}$
\end{enumerate}
We show the union case; the intersection case is completely symmetric.
By Proposition \ref{Prop3.7}, for each $n$, there is a rapidly null $G_\delta$ 
set such that on its complement, $f_{\sigma n}$ is the characteristic function  
of $|T_{\sigma n}|$.  Inspecting the proof of Proposition \ref{Prop3.7}, we 
see that the rapidly null $G_\delta$ sets guaranteed there have a uniform
$\Delta^0_1$ definition relative to the data $\langle f_{\sigma n \tau}: \tau \in T_{\sigma n}, n \in \omega\rangle$.  Let $A$ denote the rapidly null $G_\delta$ set 
obtained by combining these infinitely many tests into a single test.  Define
$$R_0^{[<\sigma]} = \bigoplus_{\substack{\tau \in T_{\sigma n} \\ n \in \omega}} R_0^{[\sigma n\tau]}.$$
Since 
$$A\leq_T \langle f_{\sigma n \tau} : \tau \in T_{\sigma n}, n \in \omega\rangle \leq_T \bigoplus_{\substack{\tau \in T_{\sigma n}\\ n \in \omega}} U_{\sigma n \tau} \in \Delta^1_1 (R^{[<\sigma]})$$
and $R^{[\sigma]}$ is $\Delta^1_1$-random relative to $R^{[<\sigma]}$, each column $R_0^{[\sigma, j]}$
avoids $A$.  Therefore, for each $p \in 2^{<\omega}$ and each $j$ and $n$, we have
$$p\concat R_0^{[\sigma,j]} \in |T_{\sigma n}|  \iff f_{\sigma n}(p \concat R_0^{[\sigma,j]}) = 1.$$
Therefore,
\begin{align*}
(p,j) \in U_\sigma &\iff p\concat R_0^{[\sigma,j]} \in |T_\sigma|\\
&\iff \exists n p\concat R_0^{[\sigma,j]} \in |T_{\sigma n}| \\
& \iff \exists n f_{\sigma n}(p \concat R_0^{[\sigma,j]}) = 1\\
&\iff \sup_n f_{\sigma n}(p \concat R_0^{[\sigma,j]})=1. \end{align*}
Here $\sup_n f_{\sigma n}$ has a canonical $L^1$ name arithmetic in 
$\langle f_{\sigma n} : n \in \omega\rangle$, and the last bi-implication 
is justified by Proposition \ref{Prop2.13}, since $p\concat R_0^{[\sigma,j]}$ 
also avoids the rapidly null $G_\delta$ guaranteed there.
 Thus by Lemma \ref{Lem4.2}, 
$$\lim_{N\rightarrow\infty} \frac{1}{N} \sum_{j<N} \sup_n f_{\sigma n}(p\concat R_0^{[\sigma,j]})$$
exists for all $\sigma$, and $\langle f_\sigma^i\rangle_{i\in \omega}$ is
actually a name for $\sup_n f_{\sigma n}$.  Therefore, by Proposition 
\ref{Prop2.13} and Corollary \ref{Cor2.12},
for almost all $x$ we have $f_\sigma(x) = \lim_i f_\sigma^i(x) = \sup_n f_{\sigma n}(x)$.  Theorem \ref{Thm2.8} then implies that $f_\sigma = \sup_n f_{\sigma n}$,
which is what we wanted to prove.
\end{proof}

\section{$\omega$-models of $\CDM$ are closed under $\Delta^1_1$-randoms}

In this section we show that any $\omega$-model $\mathcal M$
of $\CDM$ must
be closed under $\Delta^1_1$-randoms, in the sense that for every 
$Z\in \mathcal M$, there is an $R \in \mathcal M$ that is 
$\Delta^1_1$-random relative to $Z$.
We first review the machinery of decorating trees from \cite{ADMSW}.
All results summarized here relativize and 
they will be used in a relativized form, but we state them in 
unrelativized form to reduce clutter.

The purpose of the operation $\Decorate$ is to take a code for a
Borel set which may
not be completely determined, and force it to become 
determined for some ``small'' set of 
inputs, while not changing its membership 
facts for other inputs.  In our case ``small" will mean measure 0.
Roughly speaking,
we are going to make a code $T$ and add decorations to ensure that
all non $\Delta^1_1$-randoms are determined in $T$.  We will
also make sure any 
measure decomposition is complicated enough to compute 
a $\Delta^1_1$-random.  That way, if there 
are no $\Delta^1_1$-randoms then the tree 
is completely determined, at which point the existence of a 
computationally powerful
measure decomposition leads to a contradiction.

\begin{definition}[\cite{ADMSW}]
A \emph{nice decoration generator} is 
a partial computable function which 
maps any $b\in \kO^\ast$ to 
alternating, $b$-ranked 
trees $(P_b,N_b)$, where each 
$P_b$ and $N_b$ have
an intersection or a leaf at their root.  
\end{definition}

For example (and this is what we will use), there is a finite number $k$ 
such that the following almost defines a nice decoration generator.
\begin{align*}
P_{b+k} &= \{X : \text{$X$ is not $MLR^{H_b}$, but for all $c<_\ast b$, $X$ is $MLR^{H_c}$}\} \cap \{X : X <_{\text{lex}} H_b\}\\
N_{b+k} &= \{X : \text{$X$ is not $MLR^{H_b}$, but for all $c<_\ast b$, $X$ is $MLR^{H_c}$}\} \cap \{X : X \geq_{\text{lex}} H_b\}
\end{align*}
All that remains is to define $P_b$ and $N_b$ when $b$ is within $k$ 
successors of a limit ordinal; in that case we set both $P_b$ and $N_b$ 
to be $b$-ranked alternating codes for the empty set.

The operation $\Decorate$ is defined below 
using effective transfinite recursion (with parameter 
$<_\ast$ which is computable from $\emptyset'$), 
and therefore is well-defined 
on $a$-ranked trees $T$ for all $a \in \kO^{\ast,T}$.

\begin{definition}[\cite{ADMSW}] The operation $\Decorate$ is defined as 
follows.  The
inputs are an $a$-ranked labeled tree $T$
and a nice decoration generator $h$.  
\begin{align*}
\Decorate(T, h) = \{\lambda\} &\cup 
\bigcup_{\langle n \rangle \in T} 
\langle 2n \rangle\concat \Decorate(T_{\langle n \rangle}, h)\\
&\cup \bigcup_{b <_\ast \rho_T(\lambda)}
\langle 2b+1\rangle\concat\Decorate(Q_b,h)
\end{align*}
where $Q_b = P_b$ if $\lambda$ is a $\bigcup$ in $T$, and 
$Q_b = N_b^c$ if $\lambda$ is a $\bigcap$ in $T$.

The rank and label of $\lambda$ in $\Decorate(T,h)$ 
are defined to coincide with the rank and label of 
$\lambda$ in $T$.  The ranks and labels of the other nodes in 
$\Decorate(T,h)$ are inherited from $\Decorate(T_{\langle n\rangle},h)$
or $\Decorate(Q_b,h)$ as appropriate.
\end{definition}

If $T$ is $a$-ranked, so is 
$\Decorate(T,h)$.
Similarly, if $T$ and each $P_b$ and $N_b$ are 
alternating, 
then $\Decorate(T,h)$ will also be 
alternating.  (Note that in this case, $N_b^c$ 
has a union at its root).

\begin{lemma}[\cite{ADMSW}]\label{lem:kO_induction}
Let $h$ be a nice decoration generator. 
Suppose $b \in \kO$, and suppose that 
$X \not\in \makeset{P_d} \cup \makeset{N_d}$
for any $d <_\ast b$.  Then for any $b$-ranked tree $T$,
$X \in \makeset{\Decorate(T,h)}$ if and only if $X \in 
\makeset{T}$.
\end{lemma}

\begin{lemma}[\cite{ADMSW}]\label{lem:2} Let $a \in \kO^\ast$ and $b \in \kO$ 
with $b<_\ast a$.
Let $T$ be an alternating, $a$-ranked tree and 
let $h$ be a nice decoration generator.
Suppose $X \in \makeset{P_b} \cup \makeset{N_b}$.  Then
\begin{enumerate}
\item $X$ has a unique evaluation
map in $\Decorate(T,h)$.
\item This evaluation map is $H_{b + O(1)}^{X\oplus T}$-computable.
\end{enumerate}
\end{lemma}

\begin{theorem}
Suppose that $\mathcal M$ is an $\omega$-model of $\WWKL+\CDM$. 
Then for any $Z \in \mathcal M$, there is an $R \in \mathcal M$ such that $R$ is 
$\Delta^1_1$-random relative to $Z$.
\end{theorem}
\begin{proof}
If $\mathcal M$ is a $\beta$-model, then $\mathcal M$ is already 
closed under $\Delta^1_1$-randoms in the sense described above, 
because the statement $\exists R (R \text{ is }\Delta^1_1(Z)\text{-random})$
is a true $\Sigma^1_1(Z)$ statement, and any witness to its truth 
computes such an $R$.

On the other hand, if $\mathcal M$ is not a $\beta$-model, then there is a tree 
$S \in M$ such that $\mathcal M$ believes $S$ to be well-founded, but 
in fact $S$ is ill-founded.  Without loss of generality, assume that 
$Z \geq_T S$; otherwise we end up with a $\Delta^1_1$-random relative 
to $Z\oplus S$.  There is a $Z$-computable procedure which, 
given any truly well-founded tree as input, produces an element of $\kO^Z$ 
which bounds its rank.  Apply this procedure to $S$ to produce a 
pseudo-ordinal $a^\ast \in (\kO^\ast)^Z$.  Then $\mathcal M$ thinks 
that $a^\ast$ is an ordinal.  Let $T$ be any $Z$-computable,
alternating, $(a^\ast+1)$-ranked
tree such that each level-one subtree $T_n$ is $a^\ast$-ranked.  We 
can assume $T$ has a union at the root, though the symmetric choice would also work. 
Let $h$ be the nice decoration generator which produces codes for 
$P_b^Z$ and $N_b^Z$ as follows (this is just the relativized form of what 
was defined above).
\begin{align*}
P_{b+k}^Z &= \{X : \text{$X$ is not $MLR^{H_b^Z}$, but for all $c<_\ast^Z b$, $X$ is $MLR^{H_c^Z}$}\} \cap \{X : X <_{\text{lex}} H_b^Z\}\\
N_{b+k}^Z &= \{X : \text{$X$ is not $MLR^{H_b^Z}$, but for all $c<_\ast^Z b$, $X$ is $MLR^{H_c^Z}$}\} \cap \{X : X \geq_{\text{lex}} H_b^Z\}
\end{align*}
As above, we also define $P_b^Z$ and $N_b^Z$ to be $b$-ranked codes for the 
empty set in case $b$ is within $k$ successors of a limit ordinal.
  Now consider the tree
$\Decorate^Z(T,h)$.  Is it completely determined?

Suppose it is not completely determined; let $X$ be an element that does not have 
an evaluation map.  Since $\CDM+\WWKL$ implies $\LwCA$, every element of 
$HYP(X\oplus Z)$ is in $\mathcal M$.  So by Lemma \ref{lem:2}, 
for any $b \in \kO^Z$, $X\not \in |P_b^Z| \cup |N_b^Z|$  (if it were in this set, it 
would have a $HYP(X\oplus Z)$ evaluation map).  But this means that 
$X$ is $\Delta^1_1$-random relative to $Z$, since each non-random
belongs to some $|P_b^Z|\cup|N_b^Z|$.

So suppose that $\Decorate^Z(T,h)$ is completely determined.  Then by $\CDM$, it 
has a measure decomposition.  We claim that any element $R$ that 
is $1$-random relative to the measure decomposition is in fact 
$\Delta^1_1$-random relative to $Z$.  It suffices to show that the 
measure decomposition computes $H_b^Z$ for all $b \in \kO^Z$.  
Fix $b \in \kO^Z$ with $b <_\ast^Z a^\ast$ and observe that 
$\Decorate^Z(P_{b+k},h)$ appears as a level-one subtree of 
$\Decorate^Z(T,h)$.  Thus, by examining the definition of $P_{b+k}$, 
which has an intersection at the root and $\{X: X<_{\text{lex}} H_b^Z\}$ 
as a level-one subtree, we see that
 $\Decorate^Z(\{X : X<_{\text{lex}} H_b^Z\},h)$ appears 
as a level-two subtree of $\Decorate^Z(T,h)$.  (Here of course, 
$\{X : X <_{\text{lex}} H_b^Z\}$ is represented using an 
approximately $b$-ranked formula of $L_{\omega_1,\omega}$, but 
this formula contributes computational, not topological, complexity.)
Therefore, there is an $L^1$ function $f$ included 
in the measure decomposition which is equal to the characteristic 
function of $\Decorate^Z(\{X : X<_{\text{lex}} H_b^Z\},h)$ 
almost everywhere.  We claim that 
$\int f = H_b^Z$, where here we regard $H_b^Z$ as a number in $[0,1]$ 
given 
by its binary expansion.  Using $\WWKL$, 
it suffices to provide another $L^1$ function $g$
which has $\int g = H_b^Z$ and such that $g$ 
is equal to the characteristic 
function of $\Decorate^Z(\{X : X<_{\text{lex}} H_b^Z\},h)$ 
almost everywhere.  Let $g$ be the canonical measurable 
characteristic function of the open set $\bigcup_{p <_{\text{lex}} H_b^Z} [p]$.
Then by Lemma \ref{lem:kO_induction}, for any $X$ that is $MLR^{H_b^Z}$, 
since $X \not\in |P_d^Z|\cup|N_d^Z|$ for any $d<_\ast^Z b+k$, 
we have $X \in \Decorate^Z(\{X : X<_{\text{lex}} H_b^Z\},h)$ if 
and only if $X <_{\text{lex}} H_b^Z$, which is true if and only if 
$g(X) = 1$.  This completes the proof.
\end{proof}

\bibliographystyle{alpha}
\bibliography{bibliography}

\end{document}